\documentclass[11pt]{amsart}
\usepackage{amsthm}
\usepackage{amssymb}
\usepackage[usenames]{color}
\usepackage{latexsym}
\usepackage{graphicx}
\usepackage{enumerate}
\usepackage{comment}

\usepackage{comment}

\newtheorem{thm}{Theorem}[section]
\newtheorem{cor}[thm]{Corollary}
\newtheorem{lemma}[thm]{Lemma}

\newtheorem{definition}[thm]{Definition}

\title{Topological Symmetry Groups of the Heawood Graph}

\author{Erica Flapan}
\author{Emille Davie Lawrence}
\author{Robin Wilson}

\address{Department of Mathematics, Pomona College, Claremont, CA 91711}
\address{Department of Mathematics and Statistics, University of San Francisco, San Francisco, CA 94117}
\address{Department of Mathematics, Cal Poly Pomona, Pomona, CA 91768}

\date{\today}


\thanks{The first author was partially supported by NSF Grant DMS-1607744}

\begin{document}

\begin{abstract} We classify all groups which can occur as the topological symmetry group of some embedding of the Heawood graph in $S^3$.
\end{abstract}

\maketitle

\section{Introduction}

Topological symmetry groups were originally introduced in order to classify the symmetries of non-rigid molecules. In particular, the symmetries of rigid molecules are represented by the {\it point group}, which is the group of rigid motions of the molecule in space. However, non-rigid molecules can have symmetries which are not included in the point group. The symmetries of such molecules can instead be represented by the subgroup of the automorphism group of the molecular graph which are induced by homeomorphisms of the graph in space. In this way, the molecular graph is treated as a topological object, and hence this group is referred to as the {\it topological symmetry group} of the graph in space. 

Although, initially motivated by chemistry, the study of topological symmetry groups of graphs embedded in $S^3$ can be thought of as a generalization of the study of symmetries of knots and links.  Various results have been obtained about topological symmetry groups in general (\cite{MR3049299}, \cite{MR2931423}, \cite{TSG1}, \cite{TSG2},  \cite{MR3008889}, \cite{MR2909628}) as well as topological symmetry groups of embeddings of particular graphs or families of graphs in $S^3$ (\cite{MR3230775}, \cite{MR2754164}, \cite{MR3312620}, \cite{MR3208292}, \cite{MR2821428}, \cite{MR3104921}, \cite{MR3543135}, \cite{MR3264518}, \cite{MR1451857}).

In this paper, we classify the topological symmetry groups of embeddings of the Heawood graph in $S^3$. This graph, denoted by $C_{14}$, is illustrated in Figure~\ref{Heawood}.  The Heawood graph is of interest to topologists because it is obtained from the intrinsically knotted graph $K_7$ by $\Delta-Y$ moves \cite{MR2911083}, and hence is itself intrinsically knotted.  This means that every embedding of $C_{14}$ in $S^3$ contains a non-trivial knot. It also follows from \cite{MR3181638} that $C_{14}$ is intrinsically chiral, and hence no embedding of $C_{14}$ in $S^3$ has an orientation reversing homeomorphism.

\begin{figure}[h!]
\begin{center}
\includegraphics[width=.3\textwidth]{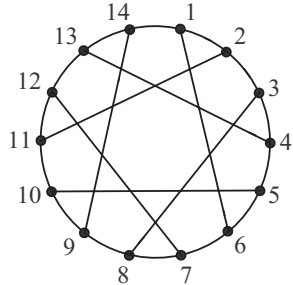}
\caption{The Heawood graph, which we denote by $C_{14}$.}
\label{Heawood}
\end{center}
\end{figure}

We begin with some terminology.

\begin{definition}  Let $\Gamma$ be a graph embedded in $S^3$.  We define the {\bf topological symmetry group} $\mathrm{TSG}(\Gamma)$ as the subgroup of the automorphism group $\mathrm{Aut}(\Gamma)$ induced by homeomorphisms of $(S^3,\Gamma)$.   We define the {\bf orientation preserving topological symmetry group} $\mathrm{TSG}_+(\Gamma)$ as the subgroup of $\mathrm{Aut}(\Gamma)$ induced by orientation preserving homeomorphisms of $(S^3,\Gamma)$. \end{definition}

\begin{definition}  Let $G$ be a group and let $\gamma$ denote an abstract graph.  If there is some embedding $\Gamma$ of $\gamma$ in $S^3$ such that $\mathrm{TSG}(\Gamma)= G$, then we say that $G$ is {\bf  realizable} for $\gamma$.  If there is some embedding $\Gamma$ of $\gamma$ in $S^3$ such that $\mathrm{TSG}_+(\Gamma)= G$, then we say that the group $G$ is {\bf positively realizable} for $\gamma$.\end{definition}

\begin{definition} Let $\varphi$ be an automorphism of an abstract graph $\gamma$.  We say $\varphi$ is {\bf realizable} if for some embedding $\Gamma$ of $\gamma$ in $S^3$, the automorphism $\varphi$ is induced by a homeomorphism of $(S^3,\Gamma)$.  If such a homeomorphism exist which is orientation preserving, then we say $\varphi$ is {\bf positively realizable}.\end{definition}

Since the Heawood graph is intrinsically chiral, a group is realizable if and only if it is positively realizable.  Our main result is the following classification theorem.

\begin{thm} \label{realizable} 
A group $G$ is realizable as the topological symmetry group for $C_{14}$ if and only if $G$ is the trivial group, $\mathbb{Z}_2$, $\mathbb{Z}_3$, $\mathbb{Z}_6$, $\mathbb{Z}_7$, $D_3$, or $D_7$. 
 \end{thm}

In Section 2, we present some background material about $C_{14}$.  In Section 3, we determine which of the automorphisms of $C_{14}$ are realizable.  We then use the results of Section 3 to prove our main result in Section 4. 

\medskip

\section {Background about the Heawood graph}

We will be interested in the action of automorphisms of $C_{14}$ on cycles of particular lengths.   The graph $C_{14}$ has 28 $6$-cycles, its shortest cycles, and 24 $14$-cycles \cite{MR0038078}, \cite{MR1892690}. The following results about the $12$-cycles and $14$-cycles are proved in the paper \cite{1907.11978}.  While some of these results may be well known, the authors could not find proofs in the graph theory literature.

\begin{lemma}\label{C14Facts} \cite{1907.11978}  

\begin{enumerate}
\item $C_{14}$ has $56$ $12$-cycles. 

\item $\mathrm{Aut}(C_{14})$ acts transitively on the set of $14$-cycles and the set of $12$-cycles.  

\item The graph obtained from $C_{14}$ by removing any pair of vertices which are a distance $3$ apart has exactly two $12$-cycles. 
\end{enumerate}
\end{lemma}

\medskip

By part (2) of Lemma~\ref{C14Facts}, we can assume that any $14$-cycle in $C_{14}$ looks like the outer circle in Figure~\ref{Heawood} and any $12$-cycle looks like the outer circle in Figure~\ref{12cycle}.  We will always label the vertices of $C_{14}$ either as in Figure~\ref{Heawood} or as in Figure~\ref{12cycle}.

  \begin{figure}[h!]
\begin{center}
\includegraphics[width=5cm]{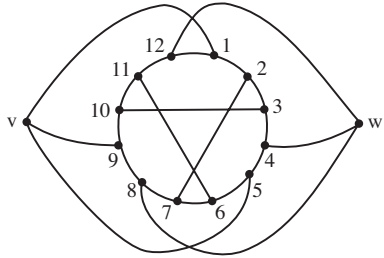}
\caption{Any $12$-cycle looks like the outer circle in this illustration.}
\label{12cycle}
\end{center}
\end{figure}

The automorphism group of $C_{14}$ is isomorphic to the projective linear group $\mathrm{PGL}(2,7)$ whose order is $336=2^4\times 3\times 7$ (see \cite{MR0038078}, \cite{MR0104735},  \cite{MR586434} ).  The program Magma was used to determine that all of the non-trivial elements of $\mathrm{PGL}(2,7)$ have order $2$, $3$, $4$, $6$, $7$, and $8$.  The following lemma gives us information about the action of automorphisms with order $3$ and $7$ on the $12$-cycles and $14$-cycles of $C_{14}$.

\medskip

\begin{lemma} \label{order3or7}  Let $\alpha$ be an automorphism of $C_{14}$.  Then the following hold.
\begin{enumerate}

\item If $\alpha$ has order $7$, then $\alpha$ setwise fixes precisely three $14$-cycles, rotating each by $\frac{2\pi n}{7}$ for some $n< 7$.

\item If $\alpha$ has order $3$, then $\alpha$ fixes precisely two vertices and setwise fixes precisely two $12$-cycles in their complement, rotating each by $\pm\frac{2\pi }{3}$.

\end{enumerate}
 \end{lemma}

\begin{proof}\noindent (1) Suppose that the order of $\alpha$ is $7$.  Since $C_{14}$ has 24 $14$-cycles, $\alpha$ must setwise fix at least three of them.  Observe that any $14$-cycle which is setwise fixed by $\alpha$ must be rotated by $\frac{2\pi n}{7}$ for some $n< 7$.  Thus every edge must be in an orbit of size 7.  Since there are 21 edges, there are precisely three such edge orbits.  Now any $14$-cycle which is setwise fixed must be made up of two of these three edge orbits, and hence there are at most three $14$-cycles which are invariant under $\alpha$.   It follows that there are precisely three invariant $14$-cycles. 

\medskip 

 \noindent (2) Suppose that the order of $\alpha$ is $3$.  Since there are $14$ vertices, $\alpha$ must fix at least two vertices $v$ and $w$. Furthermore, since $C_{14}$ has 56 $12$-cycles (by part (1) of Lemma~\ref{C14Facts}), $\alpha$ must setwise fix at least two $12$-cycles.  If some vertex on an invariant $12$-cycle were fixed, the entire $12$-cycle would be fixed and hence $\alpha$ could not have order $3$.  Thus neither $v$ nor $w$ can be on an invariant $12$-cycle.  By part (2) of Lemma~\ref{C14Facts}, we can assume that one of the invariant $12$-cycles is the outer circle in Figure~\ref{12cycle}, and hence $v$ and $w$ are as in  Figure~\ref{12cycle}.  Since $v$ and $w$ are a distance $3$ apart, it follows from part (3) of Lemma~\ref{C14Facts} that there are precisely two $12$-cycles in the complement of $\{v,w\}$.   Therefore, $\alpha$ must rotate each of the two $12$-cycles in the complement of $\{v,w\}$ by $\pm\frac{2\pi }{3}$.  \end{proof}

\medskip

\begin{lemma} \label{noReflection}  Let $\alpha$ be an order $2$ automorphism of $C_{14}$ which setwise fixes a $12$-cycle or a $14$-cycle. Then no vertex is fixed by $\alpha$.
\end{lemma} 

\begin{proof}   First suppose $\alpha$ setwise fixes a $14$-cycle and fixes at least one vertex.  Then without loss of generality, $\alpha$ setwise fixes the outer circle $C$ in Figure~\ref{Heawood} and fixes vertex $1$.  It follows that either $\alpha$ interchanges vertices $2$ and $14$ or fixes both.  In the latter case $\alpha$ would be the identity.  Thus we can assume that $\alpha$ interchanges vertices $2$ and $14$.  But since vertex $6$ is also adjacent to vertex $1$, it must also be fixed by $\alpha$. This implies that $\alpha$ interchanges the two components of $C-\{1,6\}$.  However, this is impossible because one component of $C-\{1,6\}$ has four vertices while the other has eight vertices.  

Next suppose that $\alpha$ setwise fixes a $12$-cycle.  Then without loss of generality, $\alpha$ setwise fixes the outer circle $D$ in Figure~\ref{12cycle}. Then $\alpha(\{v,w\})=\{v,w\}$.  But every vertex on $D$ has precisely one neighbor on $D$ which is adjacent to $\{v,w\}$. Thus if $\alpha$ fixed any vertex on $D$, it would have to fix every vertex on $D$, and hence would be the identity.  Now suppose $\alpha$ fixes $v$.  Since $\alpha$ has order $2$ and $v$ has three neighbors on $D$, one of these neighbors would have to be fixed by $\alpha$.  As we have already ruled out the possibility that $\alpha$ fixes a vertex on $D$, this again gives us a contradiction. \end{proof}

\medskip

\section{Realizable automorphisms of $C_{14}$}
\begin{lemma}  \label{rigid} Let $\alpha$ be a realizable automorphism of $C_{14}$.   Then the following hold.

\begin{enumerate}
\item For some embedding $\Gamma$ of $C_{14}$ in $S^3$, $\alpha$ is induced by an orientation preserving homeomorphism $h: (S^3,\Gamma) \to  (S^3,\Gamma)$ with $\mathrm{order}(h)=\mathrm{order}(\alpha)$.

\medskip

\item If $\mathrm{order}(\alpha)$ is a power of $2$, then $\alpha$ leaves at least two $14$-cycles or at least two $12$-cycles setwise invariant, and if $\mathrm{order}(\alpha)=2$, then $\alpha$ fixes no vertices.

\medskip
\item If $\mathrm{order}(\alpha)$ is even, then $\mathrm{order}(\alpha)=2$ or $6$.

\end{enumerate}
\end{lemma}

\begin{proof} (1)  Since $\alpha$ is realizable, there is some embedding $\Lambda$ of $C_{14}$  in $S^3$ such that $\alpha$ is induced by a homeomorphism $g: (S^3,\Lambda) \to  (S^3,\Lambda)$.   Now by Theorem 1 of ~\cite{MR1347360}, since $C_{14}$ is $3$-connected, there is an embedding $\Gamma$ of $C_{14}$ in $S^3$ such that $\alpha$ is induced by a finite order homeomorphism $h: (S^3,\Gamma) \to  (S^3,\Gamma)$.  Furthermore, it follows from ~\cite{MR3181638} that no embedding of $C_{14}$ in $S^3$ has an orientation reversing homeomorphism.  Thus $h$ is orientation preserving.

 Let $\mathrm{order}(\alpha)=p$ and $\mathrm{order}(h)=q$. Since $h^q$ is the identity, $p\leq q$.  If $p<q$, then $h^p$ pointwise fixes $\Gamma$, yet $h^p$ is not the identity.  However, by Smith Theory~\cite{MR0013304}, the fixed point set of $h^p$ is either the empty set or $S^1$.  But this is impossible since $\Gamma$ is contained in the fixed point set of $h^p$.  Thus $\mathrm{order}(h)=\mathrm{order}(\alpha)$. 
 
\medskip
 
\noindent (2)  Suppose that $\mathrm{order}(\alpha)$ is a power of $2$.  Let $h$ be given by part (1).  Then $\mathrm{order}(h)$ is the same power of $2$.  Let $S_1$ and $S_2$ denote the sets of $12$-cycles and $14$-cycles, respectively.  By Nikkuni  ~\cite{Nik}, for any embedding of $C_{14}$ in $S^3$, the mod $2$ sum of the arf invariants of all $12$-cycles and $14$-cycles is $1$. So an odd number of cycles in $S_1 \cup S_2$ have arf invariant $1$.  Hence for precisely one $i$, the set $S_i$ has an odd number of cycles with arf invariant $1$. Since $|S_1|=56$ and $|S_2|=24$ are each even, $S_i$ must have an odd number of cycles with arf invariant $0$ and an odd number of cycles with arf invariant 1.

 We know that $h(S_i)=S_i$ and $h$ preserves arf invariants.  Hence $h$ setwise fixes $T_0$ the set of cycles in $S_i$ with arf invariant $0$ and $T_1$ the set of cycles in $S_i$ with arf invariant $1$.  Since $\mathrm{order}(h)$ is a power of $2$, and $|T_0|$ and  $|T_1|$ are each odd, $h$ setwise fixes at least one cycle in $T_0$ and at least one cycle in $T_1$.  Hence at least two $12$-cycles or at least two $14$-cycles are setwise fixed by $h$, and hence by $\alpha$.  It now follows from Lemma~\ref{noReflection} that if $\mathrm{order}(\alpha)=2$, then $\alpha$ fixes no vertices.
 \medskip

\noindent (3) Suppose that $\mathrm{order}(\alpha)$ is even and $\mathrm{order}(\alpha)\not =2,6$.  Recall that every even order automorphism of $C_{14}$ has order $2$, $4$, $6$ or $8$. Then by part (2), $\alpha$ setwise fixes a $12$-cycle or $14$-cycle. If $\alpha$ setwise fixes a $14$-cycle, then $\mathrm{order}(\alpha)=2$ since $\mathrm{order}(\alpha)$ is even and cannot be $14$.  Thus we suppose that $\alpha$ setwise fixes a $12$-cycle $Q$, and hence $\mathrm{order}(\alpha)\not=8$

Since $\mathrm{order}(\alpha)\not =2,6$, we must have $\mathrm{order}(\alpha)=4$. Without loss of generality we can assume that  $Q$ is the outer cycle in Figure~\ref{12cycle} and $\alpha|Q=(1,4,7, 10)(2,5,8,11)(3,6,9,12)$.  But this is impossible because $\alpha(\{v,w\})=\{v,w\}$, and hence $\alpha$ cannot take vertex $4$ (which is adjacent to $w$) to vertex $7$ (which is adjacent to neither $v$ nor $w$).  Thus $\mathrm{order}(\alpha)\not=4$. \end{proof}

\medskip

\begin{thm} \label{Conclusion} A non-trivial automorphism of $C_{14}$ is realizable if and only if it has order $2$, $3$, $6$ or $7$.\end{thm}

\begin{proof} Figure~\ref{D6} illustrates an embedding of $C_{14}$ with vertices labeled as in Figure~\ref{12cycle} where vertex $w$ is at $\infty$ and the grey arrows are the edges incident to $w$. This embedding has a glide rotation $h$ obtained by rotating the picture by $\frac{2\pi}{3}$ around a vertical axis going through vertices $v$ and $w$ while rotating by $\pi$ around the circular waist of the picture. Then $h$ induces the order $6$ automorphism $(v,w)(10,11, 6, 7, 2, 3)(1,4, 9, 12, 5, 8)$.  Now $h^3$ and $h^2$ induce automorphisms of order $2$ and $3$ respectively. Thus automorphisms of orders $2$, $3$, and $6$ are realizable.
 
\begin{figure}[h!]
\begin{center}
\includegraphics[width=5cm]{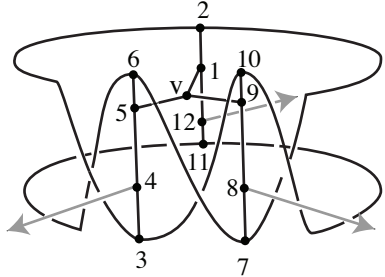}
\caption{This embedding has a glide rotation inducing $(v,w)(10,11, 6, 7, 2, 3)(1,4, 9, 12, 5, 8)$.}
\label{D6}
\end{center}
\end{figure}

Figure~\ref{Z7} shows an embedding of $C_{14}$ with a rotation of order $7$ about the center of the picture.  Thus $C_{14}$ has realizable automorphisms of order $2$, $3$, $6$, and $7$, as required.

\begin{figure}[h!]
\begin{center}
\includegraphics[width=4cm]{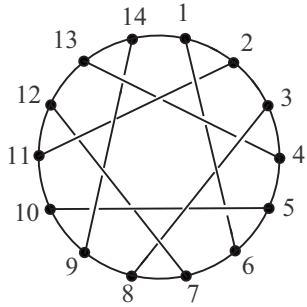}
\caption{This embedding has a rotation of order $7$.}
\label{Z7}
\end{center}
\end{figure}

For the converse, we know that that the only odd order automorphisms of $C_{14}$ have order $3$ or $7$, and part (3) of Lemma~\ref{rigid} shows that the only realizable even order automorphisms of $C_{14}$ have order $2$ or $6$.  \end{proof}\medskip

\medskip

\section{Topological symmetry groups of embeddings of $C_{14}$}

Since $C_{14}$ is intrinsically chiral  \cite{MR3181638}, for any embedding $\Gamma$ of $C_{14}$ in $S^3$, $\mathrm{TSG}(\Gamma)=\mathrm{TSG}_+(\Gamma)$.  Thus a finite group $G$ is realizable for $C_{14}$ if and only if $G$ is positively realizable.

 Let $\Gamma$ be an embedding of $C_{14}$ in $S^3$.  We know that $\mathrm{TSG}(\Gamma)$ is a subgroup of $\mathrm{Aut}(C_{14})\cong \mathrm{PGL}(2,7)$.  According to Cameron, Omidi, Tayfeh-Rezaie ~\cite{MR2240756}, the nontrivial proper subgroups of $\mathrm{PGL}(2,7)$ are $\mathbb{Z}_2$,  $\mathbb{Z}_3$, $\mathbb{Z}_4$, $\mathbb{Z}_6$, $\mathbb{Z}_7$, $\mathbb{Z}_8$,  $D_2$, $D_3$, $D_4$, $D_6$, $D_7$, $D_8$, $A_4$, $S_4$, $\mathrm{PSL}(2,7)$, $\mathbb{Z}_7 \rtimes \mathbb{Z}_3$, and $\mathbb{Z}_7 \rtimes \mathbb{Z}_6$.  We can eliminate the groups  $\mathbb{Z}_4$, $\mathbb{Z}_8$, $D_4$, $D_8$, $S_4$, $\mathrm{PSL}(2,7)$, and $\mathrm{PGL}(2,7)$ as possibilities for $\mathrm{TSG}(\Gamma)$ because we know from Theorem~\ref{Conclusion} that no realizable automorphism of $C_{14}$ has order $4$.  Thus the only groups that are possibilities for $\mathrm{TSG}(\Gamma)$ for some embedding $\Gamma$ of $C_{14}$ are the trivial group, $\mathbb{Z}_2$, $\mathbb{Z}_3$, $\mathbb{Z}_6$, $\mathbb{Z}_7$, $D_2$, $D_3$, $D_6$, $D_7$, $A_4$, $\mathbb{Z}_7 \rtimes \mathbb{Z}_3$, and $\mathbb{Z}_7 \rtimes \mathbb{Z}_6$. 

\medskip

\begin{thm} \label{classification} The trivial group and the groups $\mathbb{Z}_2$, $\mathbb{Z}_3$, $\mathbb{Z}_6$, $\mathbb{Z}_7$, $D_3$, and $D_7$ are realizable for $C_{14}$.  \end{thm}

In order to prove Theorem~\ref{classification}, we will use the following prior result.  

\begin{thm}\label{Subgroups}(\cite{MR2931423})  Let $\gamma$ be a $3$-connected graph embedded in $S^3$ as $\Gamma$ which has an edge $e$ that is not pointwise fixed by any non-trivial element of $G=\mathrm{TSG_+}(\Gamma)$.  Then every subgroup of $G$ is positively realizable for $\gamma$. 
\end{thm}

\begin{proof}[Proof of Theorem~\ref{classification}] We begin with the embedding $\Gamma$ of $C_{14}$ illustrated in Figure~\ref{D7} where the grey squares represent the same trefoil knot.  
\begin{figure}[h!]
\begin{center}
\includegraphics[width=4cm]{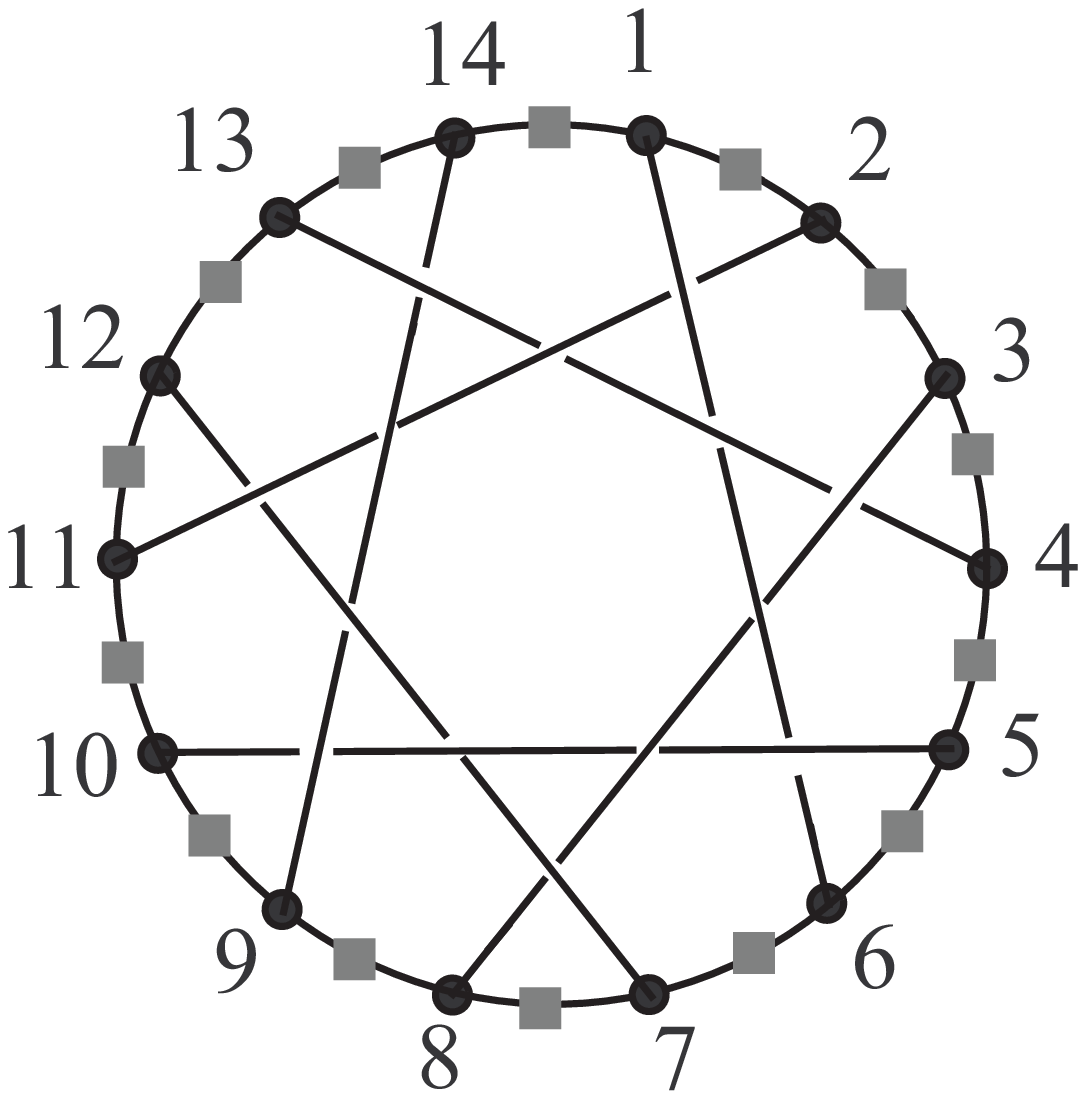}
\caption{$\mathrm{TSG}(\Gamma)=D_7$.}
\label{D7}
\end{center}
\end{figure}
The outer circle $C$ is setwise invariant under any homeomorphism of $(S^3,\Gamma)$ because $C$ is the only $14$-cycle with $14$ trefoil knots, and by \cite{MR2931423} any such homeomorphism must preserve the set of knotted edges.  It follows that  $\mathrm{TSG}(\Gamma)\leq D_{14}$.  Also, $\Gamma$ is invariant under a rotation by $\frac{2\pi}{7}$ inducing the automorphism $(1,3,5,7,9,11,13)(2,4,6,8, 10,12)$ and a homeomorphism turning $C$ over inducing $(1, 14)(2, 13)(3,12) (4, 11)(5,10)(6, 9)( 7, 8)$.  Thus $D_7\leq \mathrm{TSG}(\Gamma)$.  But $D_7$ is the only subgroup of $D_{14}$ containing $D_7$ which has no element of order $14$.  Thus $\mathrm{TSG}(\Gamma)=D_7$.

Observe that no edge of $\Gamma$ is pointwise fixed by any non-trivial element of $\mathrm{TSG}(\Gamma)=\mathrm{TSG}_+(\Gamma)$.  Hence by Theorem~\ref{Subgroups}, every subgroup of $\mathrm{TSG}(\Gamma)$ is realizable.  In particular, the groups $D_7$, $\mathbb{Z}_7$, $\mathbb{Z}_2$, and the trivial group are each realizable for $C_{14}$.

In the embedding $\Gamma'$ illustrated in Figure~\ref{D3}, $v$ is above the plane of projection, $w$ is below the plane, and the three grey squares represent the same trefoil knot.  Now $C=\overline{1, 12, 5,4, 9,8}$ is the only $6$-cycle containing three trefoil knots.  It follows that any homeomorphism of $(S^3,\Gamma')$ must take $C$ to itself taking the set of three trefoils to itself.    Thus $\mathrm{TSG}(\Gamma')\leq D_{3}$.  Since $\Gamma'$ is invariant under a $\frac{2\pi}{3}$ rotation as well as under turning the picture over, $\mathrm{TSG}(\Gamma')=D_3$.  Now if we replace the three trefoils on $C$ by three identical non-invertible knots, we will get an embedding $\Gamma''$ such that  $\mathrm{TSG}(\Gamma'')=\mathbb{Z}_3$.

\begin{figure}[h!]
\begin{center}
\includegraphics[width=5cm]{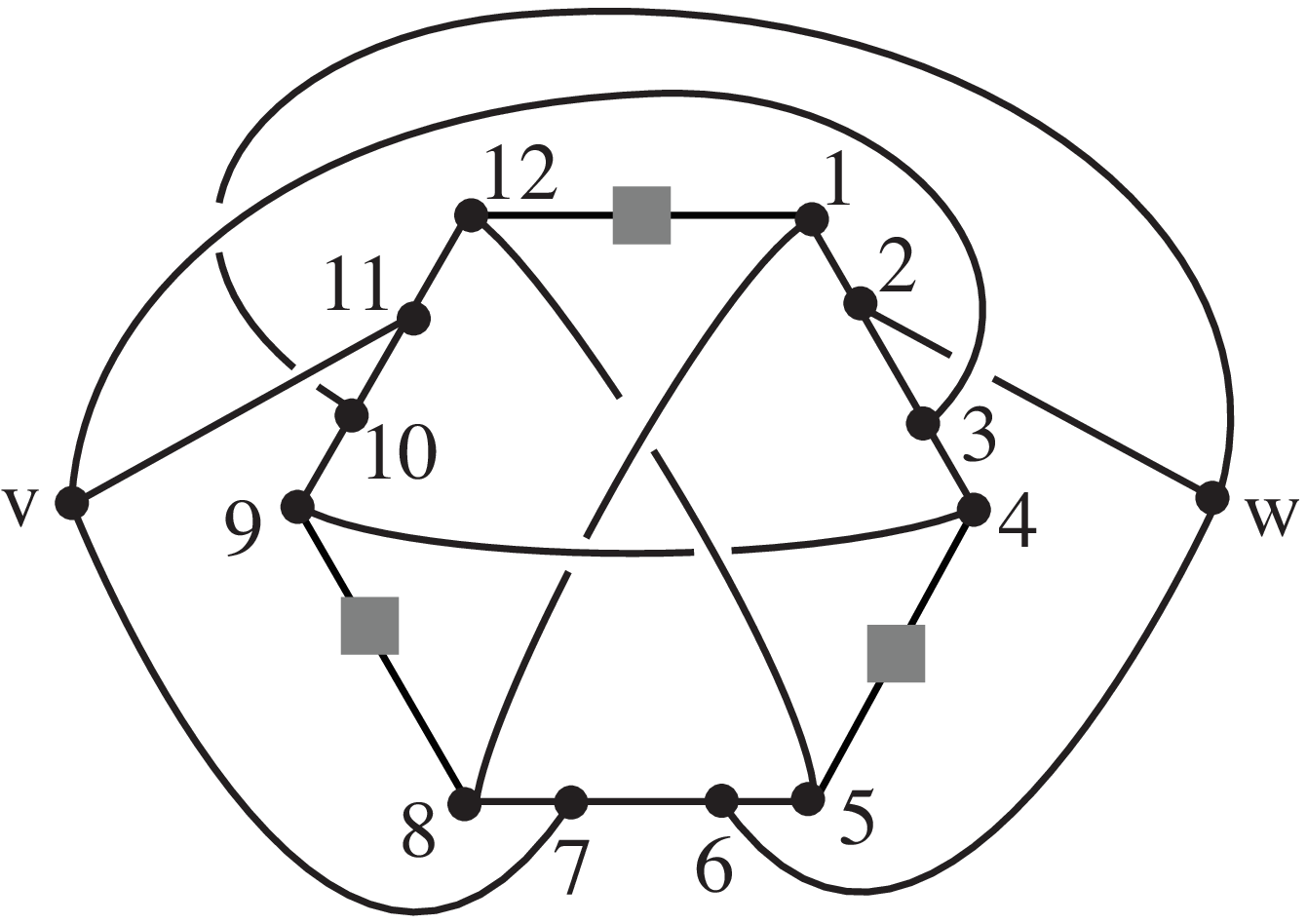}
\caption{$\mathrm{TSG}(\Gamma')=D_3$ and $\mathrm{TSG}(\Gamma'')=\mathbb{Z}_3$.}
\label{D3}
\end{center}
\end{figure}

Finally, let $\Lambda$ be the embedding in Figure~\ref{D6}.  Then the $6$-cycle $C=\overline{10,11,6,7,2,3}$ is the only $6$-cycle which contains a trefoil knot.  Thus $C$ is setwise invariant under any homeomorphism of $(S^3, \Lambda)$.  Hence $\mathrm{TSG}(\Lambda)\leq D_{6}$. We also saw in Figure~\ref{D6} that a glide rotation of $S^3$ induces the order $6$ automorphism $(v,w)(10,11,6,7,2,3)(1,4,9,12,5,8)$.  Thus $\mathbb{Z}_6\leq \mathrm{TSG}(\Lambda)$.  Since we know from Theorem~\ref{noD2} that $D_6$ is not realizable for $C_{14}$, it follows that $\mathrm{TSG}(\Lambda)=\mathbb{Z}_{6}$.\end{proof}

\medskip

In what follows, we prove that no other groups are realizable for $C_{14}$.

\begin{thm}\label{noD2}  The groups $D_2$ and $D_6$ are not realizable for $C_{14}$.  \end{thm}

\begin{proof} Suppose that there exist realizable order $2$ automorphisms $\alpha$ and $\beta$ of $C_{14}$ such that $\langle \alpha, \beta\rangle=D_2$. 
 Since $C_{14}$ has 21 edges, $\alpha$ and $\beta$ each setwise fix an odd number of edges.  Let $E_\alpha$ denote the set of edges which are invariant under $\alpha$.  Let $\varepsilon\in E_\alpha$.  Then $\alpha(\beta(\varepsilon))=\beta(\alpha(\varepsilon))=\beta(\varepsilon)$.  Thus $\beta(\varepsilon)\in E_\alpha$.  It follows that $\beta(E_\alpha)=E_\alpha$.  However, since $E_\alpha$ has an odd number of elements and $\beta$ has order $2$, there is some edge $e\in E_\alpha$ such that $\beta(e)=e$.  Thus $\alpha$ and $\beta$ both setwise fix the edge $e$, and hence at least one of the involutions $\alpha$, $\beta$, or $\alpha\beta$ must pointwise fix $e$.  
 
 Now by Lemma~\ref{noReflection}, none of  $\alpha$, $\beta$, or $\alpha\beta$ can fix any vertex.   Thus $D_2$ is not realizable.  But since $D_6$ contains involutions $\alpha$ and $\beta$ such that $\langle \alpha, \beta\rangle=D_2$, it follows that $D_6$ also cannot be realizable for $C_{14}$. \end{proof}

\medskip

\begin{thm}\label{noA4}  The group $A_4$ is not realizable for $C_{14}$.  \end{thm}

\begin{proof}  Suppose that $\Gamma$ is an embedding of $C_{14}$ such that $\mathrm{TSG}(\Gamma)=A_4$.  According to Burnside's Lemma \cite{MR0069818}, the number of vertex orbits of $\Gamma$ under $\mathrm{TSG}(\Gamma)$ is:

$$\frac{1}{|A_4|}\sum_{\alpha\in A_4}|\mathrm{fix}(\alpha)|$$

\noindent where $|\mathrm{fix}(\alpha)|$ denotes the number of vertices fixed by an automorphism $\alpha\in \mathrm{TSG}(\Gamma)$.  Observe that $A_4$ contains eight elements of order $3$, three elements of order $2$, and no other non-trivial elements.  Now by part (2) of Lemma~\ref{order3or7}, each order $3$ automorphism fixes precisely two vertices, and by Lemma~\ref{rigid} part (2), no realizable order $2$ automorphism fixes any vertex. Thus the number of vertex orbits of $\Gamma$ under $\mathrm{TSG}(\Gamma)$ is:

$$\frac{1}{|A_4|}\sum_{\alpha\in A_4}|\mathrm{fix}(\alpha)|=\frac{1}{12}((8\cdot 2)+(3\cdot 0)+(1\cdot 14))=\frac{30}{12}.$$

\noindent As this is not an integer, $A_4$ cannot be realizable for $C_{14}$.  \end{proof}

\medskip

In order to show that the groups $\mathbb{Z}_7 \rtimes \mathbb{Z}_3$ and $\mathbb{Z}_7 \rtimes \mathbb{Z}_6$ are not realizable for $C_{14}$, we will make use of the definition and prior results below.

\begin{definition}  A finite group $G$ of orientation preserving diffeomorphisms of $S^3$ is said to satisfy the \emph{involution condition} if for every involution $g\in G$, we have $\mathrm{fix}(g)\cong S^1$ and no $h\in G$ with $h\not=g$ has $\mathrm{fix}(h)=\mathrm{fix}(g)$.
\end{definition}

\begin{thm} \cite{TSG1} \label{reembed}Let $\Gamma$ be a $3$-connected graph embedded in $S^3$ with $H=\mathrm{TSG}_+(\Gamma)$.  Then $\Gamma$ can be re-embedded in $S^3$ as $\Delta$ such that $H\leq \mathrm{TSG}_+(\Delta)$ and $H$ is induced by an isomorphic finite group of orientation preserving diffeomorphisms of $S^3$. 
\end{thm}

\begin{thm}  \cite{TSG2}  \label{SO4} Let $G$ be a finite group of orientation preserving isometries of $S^3$ which satisfies the involution condition.
\begin{enumerate}
\item If $G$ preserves a standard Hopf fibration of $S^3$, then $G$ is cyclic, dihedral, or a subgroup of $D_m\times D_m$ for some odd $m$.

\item If $G$ does not preserve a standard Hopf fibration of $S^3$, then $G$ is $S_4$, $A_4$, or $A_5$.
\end{enumerate}
\end{thm}

\medskip

\begin{thm}\label{nosemidirect}   The groups $\mathbb{Z}_7 \rtimes \mathbb{Z}_3$ and $\mathbb{Z}_7 \rtimes \mathbb{Z}_6$ are not realizable for $C_{14}$.  \end{thm}

\begin{proof}  Suppose that for some embedding $\Gamma$ of $C_{14}$ in $S^3$, $ \mathrm{TSG}_+(\Gamma)$ is $\mathbb{Z}_7 \rtimes \mathbb{Z}_3$ or $\mathbb{Z}_7 \rtimes \mathbb{Z}_6$.  In either case, $G=\mathbb{Z}_7 \rtimes \mathbb{Z}_3\leq  \mathrm{TSG}_+(\Gamma)$.  Now since $C_{14}$ is $3$-connected, we can apply Theorem \ref{reembed}, to re-embed $C_{14}$ in $S^3$ as $\Delta$ such that $G\leq \mathrm{TSG}_+(\Delta)$ and $G$ is induced by an isomorphic finite group of orientation preserving diffeomorphisms of $S^3$. However, by the proof of the Geometrization Conjecture, every finite group of orientation preserving diffeomorphisms of $S^3$ is conjugate to a group of orientation preserving isometries of $S^3$  \cite{Geometrization}.  Thus, we abuse notation and treat $G$ as a group of orientation preserving isometries of $S^3$.  

Since $G$ has no elements of order $2$, it vacuously satisfies the involution condition, and hence by Theorem~\ref{SO4}, $G$ is cyclic, dihedral, a subgroup of $D_m\times D_m$ for some odd $m$, $S_4$, $A_4$, or $A_5$.  But since $|G|=21$, it cannot be dihedral, $S_4$, $A_4$, or $A_5$.  Also, since $G\leq \mathrm{Aut}(C_{14})$ has no element of order $21$, the elements of $G$ of order $3$ and $7$ cannot commute. Thus $G$ cannot be cyclic; and since all elements of odd order in $D_m\times D_m$ commute, $G$ cannot be a subgroup of any $D_m\times D_m$. By this contradiction, we conclude that neither $\mathbb{Z}_7 \rtimes \mathbb{Z}_3$ nor $\mathbb{Z}_7 \rtimes \mathbb{Z}_6$ is realizable for $C_{14}$.
 \end{proof}\medskip

The following corollary summarizes our realizability results.

\begin{cor} A group $G$ is realizable as a topological symmetry group of $C_{14}$ if and only if $G$ is the trivial group, $\mathbb{Z}_2$, $\mathbb{Z}_3$, $\mathbb{Z}_6$, $\mathbb{Z}_7$, $D_3$, or $D_7$. 
\end{cor}

\section*{Acknowledgements}The first author was partially supported by NSF grant DMS-1607744.  This grant also helped facilitate collaborative meetings of the  three authors.  The second author would like to thank the AWM for a travel grant which partially facilitated this collaboration.  The third author would like to thank Pomona College for hosting him while he was on sabbatical during part of this collaboration.  The authors also want to thank Margaret Robinson who showed them how to compute the elements of $\mathrm{PGL}(2,7)$ using the Magma Computational Algebra System.

\bibliography{Heawood_arXiv_Oct16}
\bibliographystyle{amsplain}

\end{document}